\documentclass[a4paper,12pt]{article}
\usepackage[top=2.5cm,bottom=2.5cm,left=2.5cm,right=2.5cm]{geometry}
\usepackage{cite,amsmath,amssymb}
    \usepackage[margin=1cm,%
                font=small,%
                format=hang,%
                labelsep=period,%
                labelfont=bf]{caption}
\usepackage[algoruled,linesnumbered]{algorithm2e}
\usepackage{algorithmic}

\usepackage{amsfonts,epsf,here,graphicx}
\usepackage{latexsym,multirow,rotating}
\usepackage{pgf,tikz,subfigure}
\usepackage{epstopdf}

\newtheorem{theorem}{\bf Theorem}[section]

\newtheorem{lemma}[theorem]{\bf Lemma}

\newtheorem{observation}[theorem]{\bf Observation}
\newtheorem{definition}[theorem]{\bf Definition}

\newtheorem{problem}[theorem]{\bf Problem}

\newcommand{\qed}{\hfill $\square$ \bigskip}

\begin{document}

\baselineskip=0.30in

\begin{center}
{\LARGE \bf Resonance graphs of catacondensed even ring systems}
\bigskip \bigskip

{\large \bf Simon Brezovnik$^a$, \qquad
Niko Tratnik$^{a}$,\\ \qquad Petra \v Zigert Pleter\v sek$^{a,b}$
}
\bigskip\bigskip

\baselineskip=0.20in

{\tt simon.brezovnik2@um.si, niko.tratnik@um.si, petra.zigert@um.si}
\medskip

$^a$ University of Maribor, Faculty of Natural Sciences and Mathematics, Slovenia \\
\medskip

$^b$ University of Maribor, Faculty of Chemistry and Chemical Engineering, Slovenia\\
\medskip

\bigskip\medskip

(Received \today)

\end{center}

\noindent
\begin{center} {\bf Abstract} \end{center}
A catacondensed even ring system (shortly CERS) is a simple bipartite 2-connected outerplanar graph with all 
vertices of degree 2 or 3. In this paper, we investigate the resonance graphs (also called $Z$-transformation graphs) of CERS and firstly show that two even ring chains are resonantly equivalent iff their resonance graphs are isomorphic. As the main result, we characterize CERS whose resonance graphs are daisy cubes. In this way, we greatly generalize the result known for kinky benzenoid graphs. Finally, some open problems are also presented.
\vspace{3mm}\noindent

\baselineskip=0.30in



\section{Introduction}

Catacondensed even ring systems (CERS), which were introduced in \cite{kl-br}, are a subfamily of bipartite outerplanar graphs and have an important role in chemistry since they represent various classes of molecular graphs, i.e.\ catacondensed benzenoid graphs \cite{gucy-89}, phenylenes, $\alpha$-$4$-catafusenes \cite{cy}, cyclooctatetraenes \cite{ra}, catacondensed $C_4C_8$ systems \cite{cr-tr}, etc. In this paper, we investigate the resonance graphs of CERS, since the resonance graphs model interactions among the perfect matchings (in chemistry known as Kekul\' e structures) of a given CERS. In mathematics, resonance graphs were  introduced for benzenoid graphs by Zhang et.\ al.\ \cite{zhgu-88} under the name $Z$-transformation graphs. Moreover, this concept has been independently introduced by chemists. Later, the concept was generalized so that the resonance graph was defined and investigated also for other families of graphs \cite{che0,che1,tr-zi-3,h_zhang}.

In \cite{kl-br} it was proved that the resonance graph of a CERS is a median graph and it was described how it can be isometrically embedded into a hypercube (for the definition of a median graph see \cite{che1,kl-br}). These results were applied in \cite{br-tr-zi}, where the binary coding procedure for the perfect matchings of a CERS was established. Moreover, the concept of resonantly equivalent CERS was introduced and it was shown that if two CERS are resonantly equivalent, then their resonance graphs are isomorphic. However, the backward implication remained open and we address this problem in Section 3. More precisely, we prove that it is true for even ring chains, but not in general, since we give an example of two CERS which are not resonantly equivalent and have isomorphic resonance graphs. 

Daisy cubes were introduced in \cite{kl-mo} as a subfamily of partial cubes. It was proved in \cite{zigert} that the resonance graphs of kinky benzenoid graphs are daisy cubes. In Section 4, we generalize this result to all CERS. Moreover, we characterize the CERS whose resonance graphs are daisy cubes.

\section{Preliminaries}

An \textit{even ring system} is a simple bipartite 2-connected plane graph with all interior vertices of degree 3 and all boundary
vertices of degree 2 or 3.  An even ring system  whose inner faces are only   hexagons is called a \textit{benzenoid graph}.
\smallskip

\noindent
The \textit{inner dual} of an even ring system $G$ is a graph whose vertices are
the inner faces of $G$; two vertices are adjacent if and only if the corresponding
faces have a common edge. An even ring system is {\em catacondensed} if its
inner dual is a tree $T$; in such a case we shortly call it CERS. An inner face of a CERS is called \textit{terminal} if it corresponds to a vertex of degree one in $T$. Moreover, an even ring system is called an \textit{even ring chain} if its inner dual is a path.
\smallskip

\noindent
The graphs considered in this paper are simple, finite, and connected. The {\em distance} $d_G(u,v)$ 
between vertices $u$ and $v$ of a graph $G$ 
is defined as the usual shortest path distance. The distance between two edges $e$ and $f$ of $G$, denoted by $d_G(e,f)$, is defined as the distance between corresponding vertices in the line graph of $G$. 
\smallskip

\noindent
The {\em hypercube} $Q_n$ of dimension $n$ is defined in the following way: 
all vertices of $Q_n$ are presented as $n$-tuples $x_1x_2\ldots x_n$ where $x_i \in \{0,1\}$ for each $i \in \lbrace 1,\ldots,n \rbrace$, 
and two vertices of $Q_n$ are adjacent if the corresponding $n$-tuples differ in precisely one position. A subgraph $H$ of a graph $G$ is an  {\em isometric subgraph}
if for all $u,v\in V(H)$ it holds $d_H(u,v)=d_G(u,v)$. If a graph is isomorphic to an isometric subgraph of $G$, we say that it can be {\em isometrically embedded} in $G$. Any isometric subgraph of a hypercube is called a {\em partial cube}.
\smallskip


\noindent
A {\em 1-factor} of a graph $G$ is a
spanning subgraph of $G$ such that every vertex has degree one. The edge set of a 1-factor is called a {\em perfect matching} of $G$, which is a set of independent edges covering all vertices of $G$. In chemical literature, perfect matchings are known as Kekul\'e structures (see \cite{gucy-89} for more details). 
\smallskip

\noindent
If $F,F'$ are adjacent inner faces of a CERS $G$, then the two edges on the boundary of $F$ that
have exactly one vertex on the boundary of $F'$ are called the {\em link} from $F$ to $F'$. It was proved in \cite{kl-br} that for a given perfect matching $M$ and every link either both edges or none belong to $M$. Moreover, if $M$ is a perfect matching of $G$ such that the link from $F$ to $F'$ is contained in $M$, then we say that $G$ has the \textit{$M$-link} from $F$ to $F'$.
\smallskip

\noindent
Let $G$ be a CERS. We denote the edges lying on some face $F$ of $G$ by $E(F)$. The {\em resonance graph} $R(G)$ is the graph whose vertices are the  perfect matchings of $G$, and two perfect matchings $M_1,M_2$ are adjacent whenever their symmetric difference forms the edge set of exactly one inner face $F$ of $G$, i.e.\,$M_1 \oplus M_2 = E(F)$. In such a case we say that  $M_1$ can be obtained from $M_2$ by \textit{rotating the edges} of $F$.
\smallskip

\noindent
Finally, we repeat the algorithm from \cite{br-tr-zi} that assigns a binary code of length $n$ to every perfect matching of a CERS with exactly $n$ inner faces. For this reason, we need some additional definitions.
\smallskip

\noindent
Let $F$, $F'$, $F''$ be three inner faces of a CERS $G$ such that $F$ and $F'$ have the common edge $e$ and $F',F''$ have the common edge $f$. In this case, the triple $(F,F',F'')$ is called an \textit{adjacent triple of faces}. Moreover, the adjacent triple of faces $(F,F',F'')$ is \textit{regular} if the distance $d_G(e,f)$ is an even number and \textit{irregular} otherwise. 
\smallskip

\noindent
Let $G$ be a CERS with $n$ inner faces. Starting from an arbitrarily chosen terminal inner face $F_1$ we can assign consecutive numbers
to each inner face to get the ordering $F_1,F_2, \ldots , F_n$. 
Let $T$ be an inner dual of CERS $G$ which is a tree with $n$ vertices. The pendant vertex of $T$, which corresponds to $F_1$, is
chosen as the root of this tree. The inner faces of $G$ are
then numbered such that $F_i$ is a predecessor of $F_j$ in $T$
if and only if $i<j$. Such a numbering of inner faces is called \textit{well-ordered} and can be obtained, for example,
by the Depth-First Search algorithm (DFS) or by the
Breadth-First Search algorithm (BFS).
\smallskip

\noindent
The algorithm now reads as follows \cite{br-tr-zi}.

\begin{algorithm}[H]

    \KwIn{Graph $G$ with well-ordered numbering of inner faces $F_1,F_2, \ldots , F_n$}
    \KwOut{Binary codes for all perfect matchings of a graph $G$}
   		$S:=\{00, 01, 10\}$\\
   	 \For{k = 3, \ldots, n}{
        	$S':=\emptyset$ \\
			set $j \in \left\{ 1, \ldots, k-1\right\} $ such that $F_jF_k \in E(T)$ \\
			$i=\min\{l;F_l F_j \in E(T) \}$ \\
  \uIf{$(F_{i}, F_{j}, F_{k})$ is regular}
  {\For {each $x \in S$}{$S':=S' \cup \{x0\}$ \\
  \If{$x_j=0$}{$S':=S' \cup \{x1\}$ }
  }
  }
  \Else{\For {each $x \in S$}{$S':=S' \cup \{x0\}$ \\ 
  \If{$x_j=1$}{$S':=S' \cup \{x1\}$}
  }
  }
    $S:=S'$
  }
  \caption{Binary coding of perfect matchings of CERS}
\end{algorithm}
\bigskip

\noindent
It follows from \cite{br-tr-zi,kl-br} that this coding procedure results in an isometric embedding of the resonance graph of a CERS with $n$ inner faces into a hypercube of dimension $n$. Therefore, two perfect matchings are adjacent in the resonance graph if and only if their codes differ in precisely one position and hence, we can easily construct the resonance graph from the set of binary codes of a CERS. 
\smallskip

\noindent
Finally, we mention the following observation which will be used several times in the proofs (see \cite{br-tr-zi,kl-br} for the details). Let $M$ be a perfect matching of a CERS $G$ with $n$ inner faces and let $x$ be a binary code that corresponds to $M$. Moreover, suppose that $k \in \lbrace 2, \ldots, n \rbrace$ and that $F_j$ is the inner face with the smallest index among faces adjacent to $F_k$. Then there is an $M$-link from $F_k$ to $F_j$ if and only if $x_k=1$. 

\section{Resonantly equivalent even ring chains}

As already mentioned, it was proved in \cite{br-tr-zi} that if two CERS are resonantly equivalent, then their resonance graphs are isomorphic. In this section, we provide an example which shows that the opposite direction does not hold in general. However, we prove that it is true for even ring chains.

Firstly, we repeat some definitions and results that will be needed. If $G$ is a CERS with an inner face $F$ and the outer face $F_0$, then a connected component of the graph induced by the edges in $E(F) \cap E(F_0)$ is called a {\em boundary segment}. 
\smallskip

\noindent
\textbf{Transformation 1.} \cite{br-tr-zi} Let $G$ be a CERS and $P$ a boundary segment of $G$. A CERS $G'$ is obtained from $G$ by subdividing edges of $P$ an even number of times or reversely, smoothing an even number of vertices of $P$.

\begin{definition} \cite{br-tr-zi}
Let $G$ and $H$ be two CERS. Then $G$ is \textbf{resonantly equivalent} to $H$ if it is possible to successively apply Transformation 1 on $G$ and $H$ to obtain graphs $G'$ and $H'$, respectively, such that $G'$ and $H'$ are isomorphic. In such a case we write $G \overset{R}{\sim} H$.
\end{definition}
\noindent
For an example of two resonantly equivalent CERS, see Figure \ref{resonantly_equivalent}.

\begin{figure}[h!] 
\begin{center}
\includegraphics[scale=0.7, trim=0cm 0.5cm 0cm 0cm]{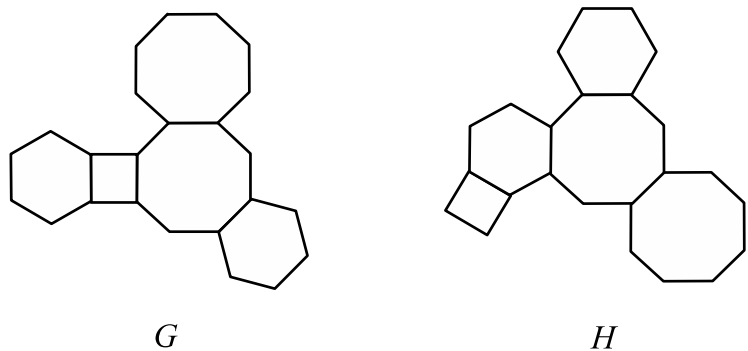}
\end{center}
\caption{\label{resonantly_equivalent} Resonantly equivalent CERS $G$ and $H$.}
\end{figure}

\noindent
The following results was proved in \cite{br-tr-zi}.

\begin{theorem} \cite{br-tr-zi} \label{glavni}
Let $G$ and $H$ be two CERS. If $G$ and $H$ are resonantly equivalent, then the resonance graph $R(G)$ is isomorphic to the resonance graph $R(H)$.
\end{theorem}

The next example shows that the backward implication of Theorem \ref{glavni} does not hold for all CERS. More precisely, let $G$ and $G'$ be the catacondensed benzenoid graphs from Figure \ref{same_resonance_graph}. Moreover, we order the inner faces $F_1, \ldots, F_8$ and $F_1', \ldots, F_8'$ of $G$ and $G'$, respectively, as shown in Figure \ref{same_resonance_graph}. It is clear that in such a way the inner faces are well-ordered. Let $k \in \lbrace 3,4,5,6,7,8 \rbrace$ and let $F_j$ be the face from the set $ \left\{ F_1,\ldots,F_{k-1} \right\}$ that is adjacent to $F_k$. Moreover, define the face $F_i$ as the face with the smallest index among all the adjacent inner faces of $F_j$. In addition, we define $F_j'$ and $F_i'$ analogously. We can easily notice that the triple $(F_i,F_j,F_k)$ is regular if and only if the triple $(F_i',F_j',F_k')$ is regular. As a consequence of this, by Algorithm 1 we obtain the same set of codes for graphs $G$ and $G'$. Hence, the resonance graphs $R(G)$ and $R(G')$ are isomorphic. On the other hand, it is obvious that the graphs $G$ and $G'$ are not isomorphic and can not be changed by Transformation 1 such that they become isomorphic. Therefore, $G$ and $G'$ are not resonantly equivalent. We have shown that for the graphs $G$ and $G'$ the backward implication of Theorem \ref{glavni} does not hold true.

\begin{figure}[h!] 
\begin{center}
\includegraphics[scale=0.7, trim=0cm 0.5cm 0cm 0cm]{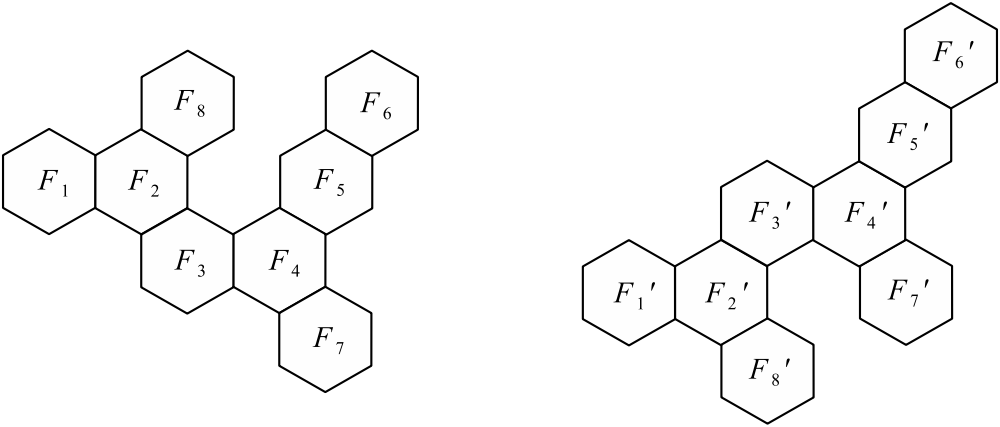}
\end{center}
\caption{\label{same_resonance_graph} Catacondensed benzenoid graphs $G$ and $G'$.}
\end{figure}

\noindent
Now we prove the main result of the section.

\begin{theorem} \label{Theorem}
If $G$ and $G'$ are two even ring chains, then $G$ and $G'$ are resonantly equivalent if and only if the resonance graphs $R(G)$ and $R(G')$ are isomorphic.
\end{theorem}

\subsection{Proof of Theorem \ref{Theorem}}

In this proof, we use some concepts and results from \cite{che1}. Therefore, the corresponding definitions (relation $\Theta$ on the edges of a graph, the graph $\Theta(R(G))$, reducible face, peripheral convex expansion) can be found in that paper.

It follows by Theorem \ref{glavni} that if $G$ and $G'$ are resonantly equivalent, then the resonance graphs $R(G)$ and $R(G')$ are isomorphic. Suppose now that $G$ and $G'$ are two even ring chains with isomorphic resonance graphs, i.e.\ $R(G) \cong R(G')$. By Theorem 3.4 in \cite{che1}, the graph $\Theta(R(G))$ is isomorphic to the inner dual $T$ of $G$ and the graph $\Theta(R(G'))$ is isomorphic to the inner dual $T'$ of $G'$. Let $\alpha: V(R(G)) \rightarrow V(R(G'))$ be an isomorphism. Obviously, $\alpha$ gives us also the corresponding isomorphism between $\Theta(R(G))$ and $\Theta(R(G'))$, which leads to the isomorphism $\beta: V(T) \rightarrow V(T')$. Therefore, $G$ and $G'$ have the same number of inner faces, let say $n$. It is clear that the theorem holds if $n=1$. Hence, we can assume $n \geq 2$.

Next, we can assign numbers to each inner face of $G$ to get the ordering $F_1,F_2, \ldots , F_n$. A terminal face is chosen as $F_1$, and then we assign the numbers to other faces such that adjacent faces get consecutive numbers. 
As a consequence, let $F_1', \ldots, F_n'$ be the corresponding ordering of inner faces of $G'$. More precisely, for each $i \in \lbrace 1, \ldots, n \rbrace$, let $F_i'=\beta(F_i)$.

By $G_k$ we denote the subgraph of $G$ induced by the faces $F_1, \ldots, F_k$, where $k \in \lbrace 1, \ldots, n \rbrace$. In addition, we define $G_k'$ in the same way. Moreover, for any $k \in \lbrace 1, \ldots, n \rbrace$, let $T_k$ and $T_k'$ be the inner duals of $G_k$ and $G_k'$, respectively. Firstly, we prove the following lemma.

\begin{lemma} \label{res_iso}
The resonance $R(G_k)$ is isomorphic to the resonance graph $R(G_k')$ for any $k \in \lbrace 1, \ldots, n \rbrace$.
\end{lemma}

\begin{proof}
Obviously, the result holds for $k=n$. Therefore, suppose that $R(G_r) \cong R(G_r')$ for some $r \in \lbrace 2, \ldots, n \rbrace$. We now show that $R(G_{r-1}) \cong R(G_{r-1}')$, which will prove the desired result.

Obviously, the inner faces $F_r$ and $F_r'$ are reducible faces in $G_r$ and $G_r'$, respectively. With other words, the vertices $F_r$ and $F_r'$ are the leaves in $T_r$ and $T_r'$. Moreover, let $E$ and $E'$ be the $\Theta$-classes in $R(G_r)$ and $R(G_r')$ corresponding to the faces $F_r$ and $F_r'$, respectively. By Theorem 3.2 in \cite{che1} we know that the resonance graph $R(G_r)$ can be obtained from $R(G_{r-1})$ by the peripheral convex expansion where the additional $\Theta$-class $E$ is obtained. On the other hand, the resonance graph $R(G_r')$ can be obtained from $R(G_{r-1}')$ by the peripheral convex expansion where the additional $\Theta$-class $E'$ is obtained. Since the $\Theta$-classes $E$ and $E'$ induce isomorphic subgraphs in $R(G_r)$ and $R(G_r')$, we obtain that the graphs $R(G_{r-1})$ and $R(G_{r-1}')$ are isomorphic. \qed
\end{proof}

\noindent
The next lemma will finish our proof.
\begin{lemma} \label{res_eq}
Graph $G_k$ is resonantly equivalent to $G_k'$ for any $k \in \lbrace 1, \ldots, n \rbrace$. 
\end{lemma}

\begin{proof}
The statement obviously holds for $k=1$ and $k=2$. Therefore, suppose that $G_r$ is resonantly equivalent to $G_r'$ for any $r$, $2 \leq r \leq n-1$. Now it suffices to show that $G_{r+1}$ and $G_{r+1}'$ are resonantly equivalent. 

Suppose that $G_{r+1}$ and $G_{r+1}'$ are not resonantly equivalent. It follows that one of the triples $(F_{r-1},F_r,F_{r+1})$, $(F_{r-1}',F_r',F_{r+1}')$ is regular and the other one is irregular. Without loss of generality we can assume that the triple $(F_{r-1},F_r,F_{r+1})$ is regular.  Next, we define
$$X_0 = \lbrace x \in V(R(G_r)) \,|\, x_r =0 \rbrace,$$
$$X_1 = \lbrace x \in V(R(G_r)) \,|\, x_r =1 \rbrace.$$

Note that in the above definition, vertices of the resonance graph are identified with their binary codes. According to Algorithm 1, we can easily see that there is a bijective correspondence between the vertices from $X_0$ and the vertices of $R(G_{r-1})$. On the other hand, the number of elements in $X_1$ is strictly smaller than the number of vertices in $R(G_{r-1})$ (since we always have some binary codes ending by 0 and some codes ending by 1). Hence $|X_1| < |X_0|$. The number of vertices $x \in V(G_{r+1})$ with $x_{r+1}=1$ equals $|X_0|$. On the other hand, the number of vertices $x \in V(G_{r+1}')$ with $x_{r+1}=1$ equals $|X_1|$. Since there is the same number of vertices in $R(G_{r+1})$ and $R(G_{r+1}')$ with $0$ in the last position, we now deduce that $|V(R(G_{r+1}))| > |V(R(G_{r+1}'))|$. Therefore, the resonance graphs $R(G_{r+1})$ and $R(G_{r+1}')$ are not isomorphic, which is a contradiction to Lemma \ref{res_iso}. We have now proven that $G_{r+1}$ and $G_{r+1}'$ are resonantly equivalent. \qed
\end{proof}

The proof of Theorem \ref{Theorem} obviously follows by Lemma \ref{res_eq}, since $G=G_n$ and $G'=G_n'$ are resonantly equivalent. \qed


\section{Resonance graphs of regular CERS are daisy cubes}

  If  $G$ is a graph and
$X \subseteq  V (G)$, then $\langle X \rangle$ denotes the subgraph of $ G$ induced by $X$. Further,  for string  $u$  of length $n$ over $B=\{0, 1\}$, i.e., $u = (u_1, \ldots, u_n) \in  B^n$
we will briefly write $u$ as $u_1\ldots  u_n$. 
Let $\leq$ be a partial order on $B^n$ defined with $u_1 \ldots u_n\leq  v_1 \ldots v_n $ if $u_i \leq v_i$ holds for all
$i \in  \lbrace 1, \ldots, n \rbrace$. For $X \subseteq B^n$ we define the graph $Q_n(X)$ as the subgraph of $Q_n$ with
$Q_n(X) = \left\langle  \{u \in B^n ; u \leq x \,\text{for some} \, x \in X\} \right\rangle $
and say that $Q_n(X)$ is a {\em daisy cube} (generated by $X$).

 The  {\em interval} $I_G(u,v)$ for two vertices $u$ and $v$ of a graph $G$ is the set of all vertices that lie on any shortest path between $u$ and $v$. The authors of the seminal paper \cite{kl-mo} on daisy cubes observed the following; if
 $\widehat{X}$ is the antichain consisting of the maximal elements of the poset $(X, \leq)$, then  $Q_n(\widehat{X})=Q_n(X)$;  further
$Q_n(X)= \left\langle \cup_{x \in \widehat{X}} I_{Q_n}(0^n,x)  \right\rangle$ (here $0^n$ denotes the string composed of $n$ zeros).\\
Note that for  a CERS $G$  with $ n$ well-ordered inner faces $F_1,\ldots, F_n$, the set of all binary labels of perfect matchings of $G$, denoted as ${\cal L}(G)$,  together
with a partial order $\leq$ is also a poset  $({\cal L}(G),\leq)$, since  ${\cal L}(G)$ is a subset of $\{0,1\}^n$.

\begin{definition} 
If a CERS $G$ has at most two inner faces or if every adjacent triple of faces of  $G$ is regular, then $G$ is a  \textbf{regular}  CERS.
\end{definition}

\noindent
The following theorem is the main result of the section.

\begin{theorem} \label{glavni_daisy}
If $G$ is a CERS, then $G$ is regular if and only if the resonance graph $R(G)$ is a daisy cube.
\end{theorem}

\subsection{Proof of the first implication of Theorem \ref{glavni_daisy}}

In this subsection we show that if $G$ is a regular CERS, then the resonance graph $R(G)$ is a daisy cube.

An inner   face $F$ of a CERS $G$ is {\em $M$-alternating} if the edges of   $E(F)$ appear
alternately in and off the perfect matching $M$. Set $S$ of disjoint inner faces of $G$ is a 
{\em resonant set} of $G$  if there exists a perfect matching $M$  such that  all faces  in $S$ are $M$-alternating. 
We can assign in a natural way the  binary label to a resonant set of a CERS.
\begin{definition}
Let $G$ be a CERS with  $n$  well-ordered inner faces  $F_i$, $i \in \lbrace 1, \ldots, n \rbrace$.
 If $S$  is a resonant set of $G$, then its \textbf{binary representation} $b(S)$ is a binary string of length $n$ where
 $$b (S)_i=\left\{ \begin{array}{rcc}  
                1 & ; & F_i\in S,\\
                0 & ; & otherwise\,.
				 \end{array}  \right.$$
\end{definition}

\noindent
In order to prove our main result, we need some lemmas.

\begin{lemma}
Let $G$ be a regular CERS, $F$ an inner face,  and $M$ a perfect matching of $G$. 
Then $F$ has the $M$-link to every adjacent inner face or to none of them. 
\label{lema0}
\end{lemma}
\begin{proof}
If $F$ has an $M$-link to some adjacent inner face $F'$ then, since $G$ is a regular CERS, every boundary segment has an even number of vertices and consequently, there is
an $M$-link from $F$  to every inner face adjacent to $F$. \qed 
\end{proof}

\begin{lemma}
Let $G$ be a regular CERS and let $S$ be a maximal resonant set
of $G$. If $ F$ is an inner face of $G$ that does not belong to $S$, then one of its
adjacent faces belongs to $S$.
\label{lema1}
\end{lemma}

\begin{proof}
 Let $G$ be a regular CERS  with   a maximal resonant set  $S$, where  $M$ is a  perfect matching of $G$ such that every face  in $S$ is $M$-alternating.
 
 Let    $F$  be  an inner face that is not in $S$ and let   $F_1, F_2,\ldots , F_k$ be the inner faces of $G$   adjacent to $F$.  For the contrary, suppose that none of the faces $F_1, F_2,\ldots , F_k$ belongs to $S$.  If none of the faces  $F_1, F_2,\ldots F_k$  has an $M$-link to $F$, then $F$ is $M$-alternating  and  we have a contradiction since $S$ is maximal. Therefore, there is the $M$-link from face $F_i$ to face $F$ for some $i=1,\ldots, k$.  By Lemma \ref{lema0}, face $F_i$ has an $M$-link to every adjacent inner face, so $F_i$ is $M$-alternating and belongs to $S$ - a contradiction. \qed 

\end{proof}

\begin{lemma}
Let $G$ be  a regular CERS and  $S$ a resonant set of $G$.   Then there exists a perfect matching $M$ of $G$ such that 
for every face $F$ from $S$  there is an $M$-link to every inner face adjacent to $F$.
\label{link}
\end{lemma}
\begin{proof}
Since $S$ is a resonant set, there exists a perfect matching $M$ such that all faces from $S$ are $M$-alternating and such that there is an $M$-link from $F$ to some inner face adjacent to it. Therefore, the result  follows directly  by Lemma \ref{lema0}. \qed  
\end{proof}

\begin{lemma}
If $G$ is a regular CERS and $S$ a resonant set of $G$, then  $b(S)$ is an element in ${\cal L}(G)$. Furthermore, if $S$ is a maximal resonant set, then $b(S)$ is a maximal element in  $({\cal L}(G),\leq)$.
\label{lema2}
\end{lemma}

\begin{proof}
Let $G$ be a regular CERS with well-ordered  inner faces $F_i$, $i  \in  \lbrace 1, \ldots, n \rbrace$, and let ${\cal L}(G)$  be the  set of binary labels of perfect matchings of $G$. First,  let $S$ be a maximal resonant set of a regular CERS   $G$ and $b(S)$ its binary label. We now show that $b(S) $ is an element of ${\cal L}(G)$ by constructing a perfect matching $M$ of $G$ such that $b(S)=\ell (M)$. By Lemma \ref{link} there exists such a perfect matching $M$ of $G$ that for every face $F$ from $S$  there is an $M$-link to every inner face adjacent to $F$. If $F_i \in S$, then it holds  $\ell (M)_{i}=1$ and so $b(S)_i=\ell (M)_{i}$.
 
On the other side, let  $F_j$ be a face that is not in $S$, so   $b(S)_j=0$.   Then, due to Lemma \ref{lema1},
at least one of the inner faces adjacent to face  $F_j$ must be in $S$, let it be face $F_{j'}$. Obviously, by the definition of perfect matching $M$ there is the $M$-link from $F_{j'}$ to $F_{j}$. Consequently, there is no $M$-link from $F_j$ to $F_{j'}$ and hence, by Lemma \ref{lema0},  there is no link from $F_j$ to any other inner face adjacent to 
it, so $\ell (M)_{j}=0$ and $b(S)_j=\ell (M)_{j}$. We have shown that $b(S)=\ell (M)$ and $b(S) \in {\cal L}(G)$.

Suppose that $S' \subset S$ and let $F=F_i$ be an inner face such that $F \in S \setminus S'$.  We define $M' = M \oplus E(F)$. Since $\ell(M)_i=1$, it holds $\ell(M')_i=0$. Obviously, for any $j \neq i$ we have $\ell(M)_j=\ell(M')_j$. After repeating the same procedure on all faces from $S \setminus S'$, we obtain the perfect matching $M^*$ such that $b(S')=\ell(M^*)$.

Finally, we have to show that $\ell (M)$ is a maximal element in $({\cal L}(G),\leq)$. For the contrary, suppose that there exists a perfect matching  $\widehat{M}$ of $G$ such that
 $\ell (\widehat{M})$  covers $\ell (M)$.  More precisely, there exists exactly one position  $j \in \lbrace 1, \ldots, n \rbrace$ such that 
 $(\ell (\widehat{M}))_j=1$ and $(\ell (M))_j=0$ and the binary labels of $\widehat{M} $ and $M$ coincide on all other positions. Then the face 
$F_j$ is not in $S$. By Lemma \ref{lema1} there exists an inner face $F_k$ adjacent to $F_j$ such that $F_k$ is in $S$.
Since $\ell(\widehat{M})_{j}=1$, we have an $M'$-link from face $F_{j}$ to some other inner face and consequently by 
Lemma \ref{lema0} to all faces adjacent to $F_j$, including face $F_k$. But then there is no $\widehat{M}$-link from $F_k$ to $F_j$ and by Lemma \ref{lema0} we have $\ell(\widehat{M})_k=0$ and hence $\ell(M)_k=0$. This leads to a contradiction since
$F_k \in S$ and so $\ell(M)_k=1$. \qed 

 \end{proof}

\begin{lemma}
Let $G$ be a regular  CERS and  $\ell(M)$  a maximal element in
$({\cal L}(G), \leq)$. Then there exists a maximal resonant set $S$ of $G$ such that $b(S)=\ell(M)$.
\label{lema3}
\end{lemma}

\begin{proof} Let $G$ be a regular CERS with  $n$ inner faces $F_i$, $i \in \lbrace 1, \ldots, n \rbrace$,
 $\ell (M)$ a maximal element in $({\cal L}(G),\leq)$  for a perfect matching of  $M$ of $G$, and $S$ a set of inner faces of $G$ such that $F_i \in S$ iff $\ell (M)_i=1$. We have to show that
$S$ is a maximal resonant set of $G$.

By the binary coding procedure and Lemma \ref{lema0} it follows that all faces in $S$ are $M$-alternating
and no two faces from $S$ are adjacent. Therefore, $S$ is a resonant set. Suppose $S$ is not a maximal resonant set, so there exists a resonant set $S'$ such that
$S \subset S'$. Let $F'$ be a face such that $F' \in S'\setminus S$.  Then there exists a perfect matching $M'$ of $G$ such
that  $S'$ is an $M'$-alternating set  and every face in $S'$, by Lemma \ref{lema0}, has a link to all faces adjacent to it
and consequently $\ell (M')_i=1$ if $F_i \in S'$. Let $F':=F_j$, so $\ell (M')_j=1$ (note that $F_j \notin S$).  
 By the definition of $S$, we have  $\ell (M)_j=0$ and therefore
 $\ell(M) \leq \ell(M')$, which  is a contradiction with the maximality of $\ell(M)$. \qed 

 \end{proof}

We now continue with the proof.
For  $X={\cal L}(G)$ let  $\widehat{X}$ be the set of all maximal elements of $({\cal L}(G),\leq)$.
By Lemma \ref{lema3}, for each $x=\ell (M) \in \widehat{X}$ there is a maximal resonant set  $S$ of $G$ such that $b(S)=\ell (M)$.
 By Lemma \ref{lema2}, for every subset $S'$ of $S$ there is a perfect matching $M'$ such that $b(S')=\ell(M')$. Therefore, if $x' \in B^n$ and $x' \leq x$, then $x' \in {\cal L}(G)$. It follows that $R(G) = \left\langle \cup_{x \in \widehat{X}} I_{Q_n}(0^n,x)  \right\rangle$, which means that $R(G)$ is a daisy cube. \qed
 
\subsection{Proof of the second implication of Theorem \ref{glavni_daisy}}

In this subsection, we show that if $R(G)$ is a daisy cube, then $G$ is a regular CERS. Some observations and lemmas are first presented.

\begin{observation} \label{ob}
If $G$ is a hypercube and $x,y,z \in V(G)$ such that $xy \in E(G)$ and $yz \in E(G)$, then there exists $u \in V(G)$, $u \neq y$, such that $u$ is adjacent to $x$ and $z$.
\end{observation}

\begin{observation} \label{ob2}
If $x_1x_2\ldots x_n$ is the binary code of some perfect matching of CERS $G$, then for any $i \in \lbrace 1, \ldots, n-1 \rbrace$ there exists such a perfect matching of $G$ that its binary code is $x_1 x_2 \ldots x_{i} 0^{n-i}$.
\end{observation}

\begin{lemma} \label{lema-a}
If $G$ is a daisy cube and $P=v_1,v_2,v_3,v_4$ a path on four vertices in $G$, then $v_1$ and $v_3$ have a common neighbour different from $v_2$; or $v_2$ and $v_4$ have a common neighbour different from $v_3$.
\end{lemma}

\begin{proof}
Let $G = Q_n(X)$, where $X \subseteq B^n$ and let $\widehat{X} = \lbrace x_1, \ldots, x_k \rbrace$ be the set of all the maximal elements of $Q_n(X)$. For any $x_i \in \widehat{X}$, let $Q_i = Q_n(\lbrace x_i \rbrace)$ be the hypercube corresponding to $x_i$, $i \in \lbrace 1, \ldots, k \rbrace$. 
Obviously, the vertices $v_1$, $v_2$ are both contained in some hypercube $Q_r$, $r \in \lbrace 1, \ldots, k \rbrace$. If $v_3 \in V(Q_r)$, then by Observation \ref{ob} there exists some vertex $u \in V(Q_r)$ different from $v_2$ such that $u$ is adjacent to $v_1$ and $v_3$ (note that $u$ can be also $v_4$). Therefore, suppose that $v_3$ does not belong to $Q_r$. We will now show that $v_2,v_3$, and $v_4$ all belong to some hypercube.
Firstly, we notice that $v_2 \leq v_3$, since otherwise $v_3$ would be in $Q_r$. Consider two options:
\begin{itemize}
\item [(a)] $v_4 \leq v_3$: in this case, $v_2$ and $v_4$ belong to every hypercube $Q_j$ for which $v_3 \in V(Q_j)$.
\item [(b)] $v_3 \leq v_4$: in this case, $v_2$ and $v_3$ belong to every hypercube $Q_j$ for which $v_4 \in V(Q_j)$.
\end{itemize}

\noindent
Therefore, there exists $j \in \lbrace 1, \ldots, k \rbrace$ such that $v_2$, $v_3$, and $v_4$ belong to $Q_j$. Hence, by Observation \ref{ob} there exists some vertex $u \in V(Q_j)$, $u \neq v_3$, such that $u$ is adjacent to $v_2$ and $v_4$. \qed
\end{proof}

\begin{lemma} \label{lema-b}
Suppose $G$ is a CERS with at least three faces such that $G$ is not a regular CERS. Then in the resonance graph $R(G)$ there exist a path $P=M_1,M_2,M_3,M_4$ on four vertices such that $M_1, M_3$ do not have a common neighbour different from $M_2$ and $M_2, M_4$ do not have a common neighbour different from $M_3$.
\end{lemma}

\begin{proof}
Since $G$ is not regular, it contains three faces $F,F',F''$ such that $F,F'$ have a common edge and $F',F''$ have a common edge and that the triple $(F,F',F'')$ is irregular. Let $F_1, \ldots, F_n$ be a well-ordering of inner faces such that $F=F_i$, $F'=F_{i+1}$, and $F''=F_{i+2}$ for some $i \in \lbrace 1, \ldots, n-2 \rbrace$ and $F_{i-1}$ (if it exists) is adjacent to $F_i$. Moreover, we can suppose that if $j < i$, then the triple $(F_j,F_{j+1},F_{j+1})$ is regular.

If $i=1$, let $x$ be a code of length 0. Moreover, for $i \geq 2$, let $x=x_1 x_2 \ldots x_{i-2} x_{i-1}$, where $x_{i-1}=0$, be the binary code of some perfect matching $M'$ of the graph $G_{i-1}$ (note that there always exists some perfect matching that its binary code ends with $0$). Then $x100, x000, x010$, and $x011$ represent the binary codes of four perfect matchings that form a path $P_4$ in the resonance graph $R(G_{i+2})$, see Figure \ref{lemma_path}. Denote this four perfect matchings by $M_1', M_2', M_3',$ and $M_4'$, respectively. By Observation \ref{ob2}, the binary codes $x10^{n-i}$, $x0^{n-i+1}$, $x010^{n-i-1}$, $x0110^{n-i-2}$ also form a path in $R(G)$. 
We now only need to show that $x10^{n-i}$ and $x010^{n-i-1}$ do not have a common neighbour different from $x0^{n-i+1}$, and that $x0^{n-i+1}$ and $x0110^{n-i-2}$ do not have a common neighbour different from $x010^{n-i-1}$. Therefore, it suffices to show that the binary codes $x110^{n-i-1}$ and $x010^{n-i-1}$ are not the vertices in $R(G)$. But this can be quickly observed as the consequence of the fact that the triple $(F_i,F_{i+1},F_{i+2})$ is irregular. \qed

\end{proof}

\begin{figure}[h!] 
\begin{center}
\includegraphics[scale=0.6, trim=0cm 0.5cm 0cm 0cm]{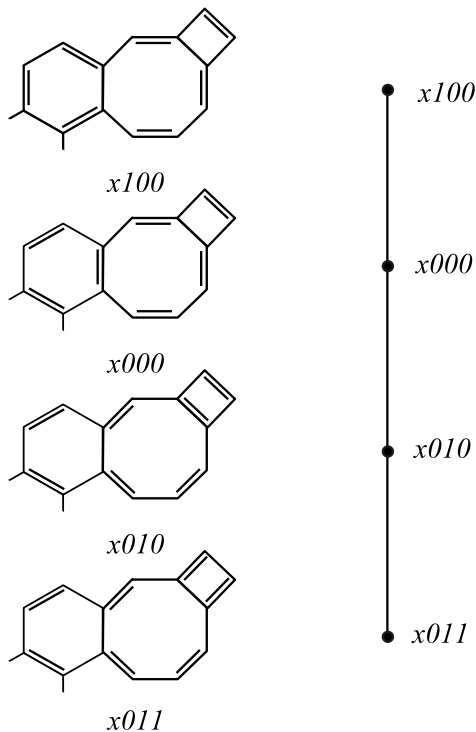}
\end{center}
\caption{\label{lemma_path} Path $P_4$ contained in the resonance graph $R(G_{i+2})$.}
\end{figure}

We finish our proof in the following way.
Suppose that $R(G)$ is a daisy cube. If $G$ has at most two inner faces, then the result is trivial by the definition of a regular CERS. Therefore, suppose that $G$ has at least three faces. If $G$ is not regular, then by Lemma \ref{lema-b} in the resonance graph $R(G)$ there exist a path $P=M_1,M_2,M_3,M_4$ on four vertices such that $M_1, M_3$ do not have a common neighbour different from $M_2$ and $M_2, M_4$ do not have a common neighbour different from $M_3$. Hence, by Lemma \ref{lema-a} the resonance graph $R(G)$ is not a daisy cube, which is a contradiction. So the proof is complete. \qed

\section{Concluding remarks and open problems}

In this paper, we have shown that two even ring chains are resonantly equivalent if and only if their resonance graphs are isomorphic. Moreover, we have proved that the resonance graph of a CERS $G$ is a daisy cube if and only if $G$ is regular. In this final section, some open problems will be presented. 

We have shown in the first part of the paper that there exist CERS $G$ and $H$ such that $R(G) \cong R(H)$ but $G$ and $H$ are not resonantly equivalent. Therefore, the following problem is interesting.

\begin{problem}
Describe a binary relation on the class of all CERS such that the following will hold: two CERS are in the relation if and only if their resonance graphs are isomorphic.
\end{problem}

In Section 4, we have characterized CERS whose resonance graphs are daisy cubes. Therefore, the same problem can be considered for all even ring systems.

\begin{problem}
Among all even ring systems (which are not catacondensed), characterize those whose resonance graphs are daisy cubes.
\end{problem}

It was shown in \cite{kl-br} that the resonance graph of any CERS is a median graph. Moreover, in \cite{vesel} the resonance graphs of catacondensed benzenoid graphs were characterized among median graphs. Therefore, the following problem appears naturally.

\begin{problem}
Among all median graphs, characterize  graphs that are resonance graphs of CERS. Moreover, find an efficient algorithm that recognizes resonance graphs of CERS.
\end{problem}

\section*{Acknowledgment} 

\noindent The authors Petra \v Zigert Pleter\v sek and  Niko Tratnik acknowledge the financial support from the Slovenian Research Agency (research core funding No. P1-0297 and J1-9109).

\end{document}